\documentclass[11pt]{amsart}
\usepackage{amsmath}
\usepackage{amsfonts}
\usepackage{amssymb}
\usepackage{latexsym}

\newtheorem{theorem}{Theorem}[section]
\newtheorem{lemma}[theorem]{Lemma}
\newtheorem{coro}[theorem]{Corollary}
\newtheorem{proposition}[theorem]{Proposition}

\newtheorem{Definition}{Definition}[section]

\newtheorem{otherth}{\bf Theorem}

\newtheorem{otherl}{\bf Lemma}

\renewcommand{\Re}{\mathrm{Re}}

\newcommand{\D}{\mathbb D}

\renewcommand{\Re}{\mbox{Re}}

\numberwithin{equation}{section}

\begin{document}

\title[Reproducing kernel estimates, projections and duality]{Reproducing kernel estimates, bounded projections and duality on large weighted Bergman spaces.}
\

\author[H. Arroussi]{Hicham Arroussi}
\address{Hicham Arroussi \\Departament de Matem\`{a}tica Aplicada i Analisi\\
Universitat de Barcelona\\ Gran Via 585 \\
08007 Barcelona\\
Spain} \email{arroussihicham@yahoo.fr}

\author[J. Pau]{Jordi Pau}
\address{Jordi Pau \\Departament de Matem\`{a}tica Aplicada i Analisi\\
Universitat de Barcelona\\ Gran Via 585 \\
08007 Barcelona\\
Spain} \email{jordi.pau@ub.edu}


\thanks{The second author was
 supported by SGR grant $2009$SGR $420$ (Generalitat de
Catalunya) and DGICYT grant MTM$2011$-$27932$-$C02$-$01$
(MCyT/MEC)}

\maketitle


\begin{abstract}
We obtain certain estimates for the reproducing kernels of large weighted Bergman spaces. Applications of these estimates to boundedness of the Bergman projection on $L^p(\D,\omega ^{p/2})$, complex interpolation and duality of weighted Bergman spaces are given.
\end{abstract}


\section{Introduction}

Let $\D=\{z\in \mathbb{C}:|z|<1\}$ be the unit disk in the complex plane $\mathbb{C}$, $dA(z) = \frac{dxdy}{\pi}$ be the normalized area measure on $\D,$  and let $H(\D)$ denote the space of all analytic functions on $\mathbb{D}.$
A weight is a positive function $\omega\in L^1(\D, dA).$ When $\omega(z) = \omega(|z|)$ for all $z\in \D,$ we say that  $\omega$ is radial.

For $0 < p < \infty ,$ the weighted Bergman space $A^p(\omega)$ is the space of all  functions $f\in H(\D)$ such that
\[
\|f \|_{A^p(\omega)} =\bigg( \int_{\D}|f(z)|^p \,\omega(z) \, dA(z)\bigg)^{\frac{1}{p}} < \infty.
\]
Our main goal is to study certain properties of the Bergman spaces $A^p(\omega)$ for a large class of weights, including certain rapidly radial decreasing weights, that is, weights that are going to decrease faster than any standard weight $(1-|z|^2)^\alpha,$ $\alpha > 0,$ such as the exponential type weights
\begin{equation}\label{Exp}
\omega_{\alpha}(z) = \exp\left(\frac{-c}{(1-|z|^2)^\alpha}\right)  ,   \  \alpha > 0 ,   \   c > 0 .
\end{equation}
For the weights $\omega$ considered in this paper, for each $z\in \D$ the point evaluations $L_ z$ are bounded linear functionals on $A^p(\omega)$. In particular, the space $A^2(\omega)$ is a reproducing kernel Hilbert space:  for each $z\in \D$, there are functions $K_ z\in A^2(\omega)$ with $\|L_ z\|=\|K_ z\|_{A^2(\omega)}$ such that $L_ z f=f(z)=\langle f, K_ z \rangle _{\omega}$, where
\[\langle f,g \rangle_{\omega}=\int_{\D} f(z)\,\overline{g(z)} \,\omega(z)\,dA(z)\]
is the natural inner product in $L^2(\D,\omega dA)$. The function $K_z$ has the property that $K_ z(\xi)=\overline{K_{\xi}(z)}$, and is called the reproducing kernel for the Bergman space $A^2(\omega).$  It is  straightforward to see from the previous formula that the orthogonal (Bergman) projection from $L^2(\D,\omega dA)$ to $A^2(\omega)$ is given by
\begin{equation}\label{BP}
 P_\omega f(z) = \int_{\D}{f(\xi) \overline{K_z (\xi)} \, \omega(\xi) \ dA(\xi)}.
\end{equation}
Some basic properties of the Bergman spaces with radial rapidly decreasing weights are not yet well understood and
have attracted some attention in recent years \cite{Ass,D2,GaPau1,KrMc,LR2,PauPel2}. The interest in such spaces arises from the fact that the usual techniques for the standard Bergman spaces fail to work in this context, and therefore new tools must be developed. For example,
the natural Bergman projection $P_{\omega}$ is not necessarily bounded on $L^p(\D,\omega dA)$ unless $p=2$ (see \cite{D1} for the exponential weights and \cite{Zey} for more examples). Another difficulty when studying these spaces arises from the lack of an explicit expression of the reproducing kernels.

It turns out that when studying properties or operators such as the Bergman projection where the reproducing kernels are involved, the most convenient setting are the spaces $A^p(\omega^{p/2})$ (or the weighted Lebesgue spaces $L^p(\omega ^{p/2}):=L^p(\D,\omega^{p/2}dA)$), and our first main result is that, for the class of weights $\omega$ considered, the Bergman projection $P_{\omega}$ is bounded from $L^p(\omega^{p/2})$ to $A^p(\omega^{p/2})$ for $1\le p<\infty$ (see Theorem \ref{projbd}).
A consequence of that result will be the identification of the dual space of $A^p(\omega^{p/2})$ with the space $A^{p'}(\omega^{p'/2})$ under the natural integral pairing $\langle \,,\,\rangle_{\omega}$, where $p'$ denotes the conjugate exponent of $p$.

The key ingredient for the obtention of the previous mentioned results is a certain integral type estimate involving the reproducing kernels $K_ z$. This integral estimate will be deduced from a pointwise estimate for $|K_ z(\xi)|$ that can be of independent interest (Theorem \ref{RK-PE}). The pointwise estimate obtained can be thought as the analogue for weighted Bergman spaces of the corresponding ones obtained by Marzo and Ortega-Cerd\`{a} \cite{M-OC} for reproducing kernels of weighted Fock spaces.

Throughout this work, the letter $C$ will denote an absolute constant whose value may change at different occurrences. We also use the notation
$a\lesssim b$ to indicate that there is a constant $C > 0$ with $a \leq C b,$ and the notation $a \asymp b$ means that $a \lesssim b$ and $b \lesssim a.$

\section{Preliminaries}

In this section we provide the basic tools for the proofs of the main results of the paper. For $a\in \D$ and $\delta>0$, we use the notation $D(\delta\tau(a))$ for the euclidian disc centered at $a$ and radius $\delta \tau(a)$.

A positive function $\tau$ on $\D$ is said to belong to the class $\mathcal{L}$ if satisfies the following two properties:
\begin{itemize}
\item[(A)] There is a constant $c_ 1$ such that $\tau(z)\le c_ 1\,(1-|z|)$ for all $z\in \D$;\\
\item[(B)] There is a constant $c_ 2$ such that $|\tau(z)-\tau(\zeta)|\le c_ 2\,|z-\zeta|$ for all $z,\zeta\in \D$.\\
\end{itemize}
 We also use the notation
 \[m_\tau : = \displaystyle \frac{\min(1, c_1^{-1}, c_2^{-1})}{4},\]
  where $c_1$ and $c_2$ are the constants appearing in the previous definition.
It is easy to see from conditions (A) and (B) (see \cite[Lemma 2.1]{PauPel1}) that if $\tau \in\mathcal{L}$  and $ z \in  D(\delta\tau(a)) ,$ then
\begin{equation}\label{asymptau}
 \frac{1}{2}\,\tau(a)\leq \tau(z) \leq  2 \,\tau(a),
\end{equation}
 for sufficiently small $\delta > 0 ,$ that is, for $\delta \in (0, m_\tau).$ This fact will be used several times in this work.

\begin{Definition}
We say that a weight $\omega$ is in the class  $\mathcal{L}^*$ if it is of the form $\omega =e^{-2\varphi}$, where $\varphi \in C^2(\D)$ with $\Delta \varphi>0$, and  $\big (\Delta \varphi (z) \big )^{-1/2}\asymp \tau(z)$, with $\tau(z)$ being a function in the class $\mathcal{L}$. Here $\Delta$ denotes the classical Laplace operator.
\end{Definition}
The following result is from \cite[Lemma 2.2]{PauPel1} and can be thought as some type of generalized sub-mean value property for $|f|^p\,\omega$ that gives the boundedness of the point evaluation functionals on $A^p(\omega)$.

\begin{otherl}\label{subHarmP}
Let $\omega \in \mathcal{L}^*$, $0<p<\infty$ and let $z\in \D$. If $\beta \in \mathbb{R}$ there exists $ M \geq 1$ such that
\[ |f(z)|^p \omega(z)^{\beta} \leq  \frac{M}{\delta^2\tau(z)^2}\int_{D(\delta\tau(z))}{|f(\xi)|^p \omega(\xi)^{\beta} \   dA(\xi)},\]
for all $ f \in H(\mathbb{D})$  and all $\delta > 0 $ suficiently small.
\end{otherl}
It can be seen from the proof given in \cite{PauPel1} that one only needs $f$ to be holomorphic in a neighbourhood of $D(\delta \tau(z))$
Another consequence of the above result is that the Bergman space $A^p(\omega ^{\beta})$ is a Banach space when $1 \leq p < \infty$ and a complete metric space when $0 < p < 1 $.\\

Since the norm of the point evaluation functional equals the norm of the reproducing kernel in $A^2(\omega),$  the result of Lemma \ref{subHarmP} also gives an upper bound for $\|K_z\|_{A^2(\omega)} $. The next result \cite{LR1} says that (at least for some class of weights) this upper bound is the corresponding growth of the reproducing kernel.
\begin{otherl}\label{RK-norm}
Let $ \omega \in \mathcal{L}^*$ and suppose that the associated function $\tau(z)$ also satisfies the condition
\begin{itemize}
\item[(C)] there are constants
$b>0$ and $0<t<1$ such that
 \[\tau(z)\le
\tau(\xi)+t|z-\xi|,\quad\text{ for }|z-\xi|>b\,\tau(\xi).\]
\end{itemize}
Then
\begin{equation*}
\|K_z\|_{A^2(\omega)}^2 \, \omega(z)  \asymp \frac{1}{\tau(z)^2},\qquad z\in \D.
\end{equation*}
\end{otherl}
The result in Lemma \ref{RK-norm} has also been obtained \cite{BDK,PauPel1} for radial weights $\omega \in \mathcal{L}^*$ for which the associated function $\tau(|z|)=\tau(z)$ decreases to $0$ as $r=|z|\rightarrow 1^{-}$, $\tau '(r)\rightarrow 0$ as $r\rightarrow 1$, and moreover, either there exist  a constant $C>0$  such that
$\tau(r)(1-r)^{-C}$ increases for $r$ close to $1$ or
 $\lim_{r\to 1^-}\tau'(r)\log\frac{1}{\tau(r)}=0.$

If a function $\tau$ satisfies the condition (C), it does not necessarily hold that $k\tau(z)$ satisfies the same condition (C) for all $k>0$ (an example of this phenomena are the standard weights $(1-|z|^2)^{\beta}$), but this is true for the exponential type weights given by \eqref{Exp}, and therefore these weights satisfies the following strongest condition
\begin{itemize}
\item[(E)] For each $m\ge 1$, there are constants
$b_ m>0$ and $0<t_ m<1/m$ such that
\[\tau(z)\le
\tau(\xi)+t_ m|z-\xi|,\quad\text{ for }|z-\xi|>b_ m\,\tau(\xi).\]
\end{itemize}
This leads us to the following definition.
\begin{Definition}
A weight $\omega$ is in the class $\mathcal{E}$ if $\omega \in \mathcal{L}^*$ and its associated function $\tau$ satisfies the condition (E).
\end{Definition}
The prototype of a weight in the class $\mathcal{E}$ are the exponential type weights given by \eqref{Exp}. An example of a non radial weight in the class $\mathcal{E}$ is given by $\omega_{p,f}(z)=|f(z)|^p\,\omega(z)$, where $p>0$, $\omega$ is a radial weight in the class $\mathcal{E}$, and $f$ is a non-vanishing analytic function in $A^p(\omega)$.\\

In order to obtain pointwise estimates for the reproducing kernels we will need the classical  H\"ormander's theorem \cite[Lemma 4.4.1]{Ho} on $L^2$-estimates for solutions of the $\overline{\partial}$-equation and also a variant due to Berndtsson \cite[Lemma 2.2]{Bern}.

\begin{otherth}[H\"ormander]\label{hormander}
Let $\varphi \in C^2(\D)$  with $\Delta\varphi > 0$ on $\D.$  Then  there exists a solution $u$ of the equation $\overline{\partial}u = f$ such that \[\int_{\D} |u(z)|^2\, e^{-2\varphi(z)} \,dA(z) \leq \int_{\D}\frac{|f(z)|^2}{\Delta\varphi(z)}\, e^{-2\varphi(z)}\, dA(z) ,\]
provided the right hand side integral is finite.
\end{otherth}
Recall that the operators $\partial$ and $\overline{\partial}$ are defined by
\begin{displaymath}
\partial:=\frac{\partial}{\partial z}=\frac{1}{2} \left (\frac{\partial}{\partial x}-i \frac{\partial}{\partial y}\right ) \quad \textrm{and}\quad \overline{\partial}:=\frac{\partial}{\partial \bar{z}}=\frac{1}{2} \left (\frac{\partial}{\partial x}+i \frac{\partial}{\partial y}\right )
\end{displaymath}
provided the use of the identification $z=x+i y$  is made. Also $\Delta=4\partial \overline{\partial}$.

\begin{otherth}[Berndtsson]\label{Bern}
Let $\varphi, \psi\in C^2(\D)$, with $\Delta\psi > 0$ on $\D.$ If
\[\bigg|  \frac{\partial\psi}{\partial w}\bigg|^2 \leq C_1 \Delta\varphi ,\,\,   with \,\,  0 <  C_1 < 1,\]
and for any $g$ one can find $v$ such that $\overline{\partial}v = g$  with
\begin{equation*}\label{BernEqua}
 \int_{\D} |v(z)|^2 \,e^{-2\varphi(z)-2\psi(z)}\, dA(z) \leq \int_{\D} \frac{|g(z)|^2}{\Delta\varphi(z)}\, e^{-2\varphi(z)-2\psi(z)} \,dA(z) ,
\end{equation*}
then for the  solution $v_0$ with minimal norm in $L^2(\mathbb{D}, e^{-2\varphi} dA),$ one has
 \[\int_{\D} |v_0(z)|^2\, e^{-2\varphi(z)+2\psi(z)}\, dA(z) \leq  C  \int_{\D} \frac{|g(z)|^2}{\Delta\varphi(z)}\, e^{-2\varphi(z)+2\psi(z) } \,dA(z) ,\]
where $C = 6/(1- C_1) .$
\end{otherth}

\section{Estimates for reproducing kernels}\label{sec2:1}

In this section we will give some  pointwise estimates  for the reproducing kernel, especially far from the diagonal, as well as an integral type estimate involving the  reproducing kernel. For weights in the class $\mathcal{L}^*$, and points close to the diagonal, one has the following well-known estimate (see \cite[Lemma 3.6]{LR2} for example)
\begin{equation}\label{RK-Diag}
|K_ z(\zeta)| \asymp \|K_ z\|_{A^2(\omega)}\cdot \|K_{\zeta}\|_{A^2(\omega)} ,\qquad \zeta \in D(\delta \tau(z))
\end{equation}
for all $\delta \in (0,m_{\tau})$ sufficiently small. Thus, the interest of the next result relies when we are far from the diagonal.

\begin{theorem}\label{RK-PE}
Let $K_z $ be the reproducing kernel of  $A^2(\omega)$ where $\omega$ is a weight in  the class $\mathcal{E}$. For each $M\ge 1$, there exists a constant $ C > 0$ (depending on $M$) such that for each  $z,\xi \in \D$ one has
\begin{equation*}
 |K_z(\xi)|\leq C \frac{1}{\tau(z)}\frac{1}{\tau(\xi)}\,\omega(z)^{-1/2}\omega(\xi)^{-1/2} \left (\frac{\min(\tau(z),\tau(\xi))}{|z-\xi|}\right )^M.
\end{equation*}
\end{theorem}
 The proof  is going to be similar to the pointwise estimate obtained by  Marzo and Ortega-Cerd\`{a}  in \cite{M-OC} for weighted Fock spaces. However, if one wants to follow the proof given in \cite{M-OC} one needs the function $\tau$ to satisfy the following condition: for all   $z,\xi\in\mathbb{D} $  with $z\notin D(\delta\tau(\xi))$ there is $\eta > 0$ sufficiently small and a constant $c_3 > 0$ such that
\begin{equation*}
\frac{\tau(\xi)}{\tau(z)} \leq c_3 \bigg( \frac{|z-\xi|}{\tau(z)}\bigg)^{1-\eta},
\end{equation*}
and it is easy to see that this condition is never satisfied in the setting of the unit disk. Instead of that, we will use a mixture of the arguments in \cite{M-OC} with the ones in \cite{OP}. Before going to the proof of Theorem \ref{RK-PE} we need
an auxiliary result.

\begin{lemma}\label{Lfi}
Let $a\in \D$. For $\varepsilon>0$ and $M\ge 1$ there is a constant $\beta$ (depending on $M$ and $\varepsilon$) such that the function
\[\varphi_ a(z)=\frac{M}{4} \log \left (1+\frac{|z-a|^2}{\beta ^2 \tau(a)^2} \right )\]
satisfies
\begin{displaymath}
\big |\partial \varphi_ a (z)\big |^2 \le \varepsilon \,\Delta \varphi(z),\quad \textrm{ and }\quad \Delta \varphi_ a(z)\le \varepsilon \, \Delta \varphi (z).
\end{displaymath}
\end{lemma}

\begin{proof}
Since $\Delta \varphi(z)\asymp \tau(z)^{-2}$ and $\varepsilon>0$ is arbitrary, it suffices to show that
\begin{displaymath}
\big |\partial \varphi_ a (z)\big |^2 \le \frac{\varepsilon}{\tau(z)^2},\quad \textrm{ and }\quad \Delta \varphi_ a(z)\le  \frac{\varepsilon}{\tau(z)^2}.
\end{displaymath}
An easy computation gives
\begin{displaymath}
\partial \varphi_ a(z)= \frac{M}{4}\, \frac{\overline{z-a}}{\beta^2 \,\tau(a)^2+|z-a|^2}.
\end{displaymath}
Therefore,
\[\Big | \partial \varphi_ a(z) \Big |^2 =\left (\frac{M}{4} \right )^2 \frac{|z-a|^2}{(\beta^2 \,\tau(a)^2+|z-a|^2)^2}.\]
Another computation yields
\begin{displaymath}
\Delta \varphi_ a(z)=4 \overline{\partial} \partial \varphi_ a(z)= M\, \frac{\beta^2 \,\tau(a)^2}{(\beta^2 \,\tau(a)^2+|z-a|^2)^2}.
\end{displaymath}
Fix $m>2M \varepsilon^{-1/2}$ and let $b>\max(m,b_ m)$, where $b_ m$ is the number appearing in condition (E).\\
\\
$\bullet$ If $|z-a|\le b \,\tau(a)$, then due to condition (B),
\begin{displaymath}
\begin{split}
\frac{1}{\tau(z)^2} \ge \frac{1}{(\tau(a)+c_ 2 |z-a|)^2}\ge \frac{1}{\tau(a)^2\,(1+bc_ 2)^2 },
\end{split}
\end{displaymath}
where $c_ 2$ is the constant appearing in condition (B). Since clearly
\[\Big | \partial \varphi_ a(z) \Big |^2 \le \left (\frac{M}{4\beta}\right )^2 \frac{1}{\tau(a)^2},\]
we only need to choose $\beta$ big enough satisfying
\begin{equation}\label{EqC-1}
\left (\frac{M}{4\beta}\right )^2 \le \varepsilon\,\frac{1}{(1+c_ 2 b)^2}
\end{equation}
to get
\[\Big | \partial \varphi_ a(z) \Big |^2 \le \varepsilon\,\tau(z)^{-2},\qquad |z-a|\le b \tau(a).\]
Also,
\[\Delta \varphi_ a(z)\le \frac{M}{\beta ^2} \frac{1}{\tau(a)^2}.\]
Hence choosing $\beta$ big enough so that
\begin{equation}\label{EqC-2}
\frac{M}{\beta ^2}\le \varepsilon\,\frac{1}{(1+c_ 2 b)^2}
\end{equation}
we have
\[\Delta \varphi_ a(z)\le \varepsilon\,\tau(z)^{-2},\qquad |z-a|\le b \tau(a).\]
Thus, to get both \eqref{EqC-1} and \eqref{EqC-2} we only need to choose $\beta>0$ satisfying
\[\beta^2>\varepsilon^{-1}\,(1+c_ 2 b)^2\,\max \Bigg (M,\Big (\frac{M}{4} \Big )^2\Bigg ).\]
$\bullet$ If $|z-a|>b \,\tau(a)$, by condition (E) we have
\[ \frac{1}{\tau(z)^2} \ge \frac{1}{(\tau(a)+t_ m |z-a|)^2}\ge \frac{1}{|z-a|^2} \cdot \frac{1}{(b^{-1}+t_ m)^2}.\]
Here $t_ m$ is the number appearing in condition (E). On the other hand,
\[|\partial \varphi_ a (z)|^2 \le \left (\frac{M}{4}\right )^2\,\frac{1}{|z-a|^2}\]
and
\[\Delta \varphi_ a (z) \le \frac{M}{|z-a|^2}.\]
Thus, we only need
\begin{equation}\label{C-M}
\max \left (M,\Big (\frac{M}{4}\Big )^2 \right ) \le \varepsilon \cdot \frac{1}{(b^{-1}+t_ m)^2}
\end{equation}
to get the result. But, since $0<t_ m<1/m$, $b>m$ and $m>2 M \varepsilon^{-1/2}$,
\begin{displaymath}
\frac{1}{(b^{-1}+t_ m)^2} >\frac{1}{(b^{-1}+1/m)^2}=\frac{m^2}{(b^{-1}m +1)^2}>\Big (\frac{m}{2} \Big )^2>M^2\,\varepsilon^{-1}
\end{displaymath}
that gives \eqref{C-M} finishing the proof.
\end{proof}

\begin{proof}[\textbf{Proof of Theorem \ref{RK-PE}}]
Let $ z,\xi \in\D$ and fix $0<\delta<m_{\tau}$. If $D(\delta\tau(z))\cap D(\delta\tau(\xi)) \neq \emptyset ,$  then
\[\frac{\min(\tau(z),\tau(\xi))}{|z-\xi|}\gtrsim 1\]
due to \eqref{asymptau}, and  therefore the result  follows from the inequality $|K_z(\xi)| \leq \|K_z\|_{A^2(\omega)}\|K_\xi\|_{A^2(\omega)}$ together with Lemma  \ref{RK-norm}.

Next suppose  that $D(\delta\tau(z))\cap D(\delta\tau(\xi)) = \emptyset. $
Let $0 \leq \chi \leq 1 $ be a function in $C_c ^\infty (\mathbb{D})$ supported on the disk $  D(\delta\tau(\xi))$\,  such that\,  $\chi \equiv 1$  in   $D(\frac{\delta}{2}\, \tau(\xi))$
and\,\, $|\overline\partial\chi|^2\lesssim \displaystyle\frac{\chi}{\tau(\xi)^2}.$  By Lemma \ref{subHarmP} we obtain
\begin{equation}\label{KRN}
\begin{split}
|K_z(\xi)|^2 \omega(\xi)&\lesssim \frac{1}{\tau(\xi)^2}\int_{D(\frac{\delta}{2}\tau(\xi))}{|K_z(s)|^2 \omega(s) \, dA(s)}
\\
&\lesssim\frac{1}{\tau(\xi)^2}\,\big \|K_z \big \|^2_{L^2(\D,\,\chi \omega dA)}.
\end{split}
\end{equation}
 By duality,
$\|K_z\|_{_{L^2(\D,\chi\omega )}} = \sup_{\substack{f}}|\langle f,K_z\rangle_{L^2(\mathbb{D},\chi\omega)}|,$  where the supremum runs over all holomorphic functions $f$ on $D(\delta\tau(\xi))$ such that
\[ \int_{D(\delta\tau(\xi))}{|f(z)|^2\omega(z)dA(z)}= 1.\]
As  $ f\chi\in L^2(\mathbb{D},\omega\,  dA)$ one has $$ \langle f,K_z\rangle_{L^2(\mathbb{D},\chi\omega)} = P_{\omega}(f\chi)(z),$$
where $P_{\omega}$ is the Bergman projection given by  \eqref{BP}
 which is obviously bounded from $L^2(\mathbb{D},\omega\, dA)$   to  $A^2(\omega).$ Now we consider $ u = f\chi - P_{\omega}(f\chi)$ the solution with minimal norm in  $L^2(\mathbb{D},\omega\, dA)$ of the equation
\begin{equation}\label{dbar-eq}
\overline\partial{u} = \overline\partial ({f\chi}) = f\overline\partial\chi.
\end{equation}
Since $\chi(z) = 0,$  we get
\[ |\langle f,K_z\rangle_{L^2(\mathbb{D},\chi\omega)}|=| P_{\omega}(f\chi)(z)| = |u(z)|.\]
Consider the subharmonic function
\[ \varphi_ {\xi}(s) = \frac{M}{4} \log \left (1+\frac{|s-\xi|^2}{\beta ^2 \tau(\xi)^2} \right ),\]
with $\beta>0$ (depending on $M$) being the one given in Lemma \ref{Lfi}, that gives
\begin{equation}\label{DELTAA2}
\big |\partial \varphi_ {\xi} (s)\big |^2 \le \frac{1}{2} \,\Delta \varphi(s),\quad \textrm{ and }\quad \Delta \varphi_ {\xi}(s)\le \frac{1}{2} \Delta \varphi (s).
\end{equation}
We will write $\tau_{\varphi}(z)$  if we need to stress the dependence on $\varphi.$
We consider the function  $ \rho = \varphi - \varphi_{\xi}.$ From   \eqref{DELTAA2},  we get  $\Delta\rho \geq \frac{1}{2}\,\Delta\varphi .$  Since it is clear that  $\Delta\varphi \geq \Delta\rho ,$ this implies
\begin{equation}\label{Eq-tau}
\tau_{\rho}(\zeta) \asymp \tau_{\varphi}(\zeta), \quad \zeta \in \mathbb{D}.
\end{equation}
 Notice that  the function $u$ is holomorphic in $ D(\delta\tau_{\rho}(z))$  for some $\delta > 0.$ Then, using  the notation $\omega_\rho = e^{-2\rho},$ by Lemma  \ref{subHarmP} and the remark following it,
\begin{equation}\label{unique}
\begin{split}
|u(z)|^2 \omega_{\rho}(z)&\lesssim  \frac{1}{\tau_{\rho}(z)^2}\int_{D(\frac{\delta}{2}\tau_{\rho}(z))}|u(s)|^2 \,\omega_\rho(s) \,dA(s)
\\
&\lesssim  \frac{1}{\tau_{\rho}(z)^2}\int_{\D} |u(s)|^2 \,\omega_\rho(s) \,dA(s).
\end{split}
\end{equation}
We know that $\Delta(\varphi+\varphi_{\xi}) > 0,$ then  applying H\"{o}rmander's Theorem \ref{hormander}, one has $v$ such that $\overline{\partial}v = \overline{\partial}(f\chi)$ with
\begin{displaymath}
\int_{\D}|v|^2 e^{-2\varphi-2\varphi_{\xi}} dA
\leq \int_{\D}\frac{|\overline{\partial}v|^2}{\Delta (\varphi + \varphi_{\xi})}e^{-2\varphi -2\varphi_{\xi}}  dA
\leq \int_{\D} \frac{|\overline{\partial}v|^2}{\Delta \varphi}e^{-2\varphi-2\varphi_{\xi}}  dA.
\end{displaymath}
This together with \eqref{DELTAA2} says that we are under the assumptions of Berndtsson theorem. Therefore, since $u$ is the solution with minimal norm in $L^2(\mathbb{D},e^{-2\varphi} dA)$ of the equation \eqref{dbar-eq},
using  Theorem \ref{Bern} we obtain
\begin{displaymath}
\int_{\D} |u(s)|^2 e^{-2\varphi(s)+ 2\varphi_{\xi}(s)}dA(s)\lesssim \int_{\D} \frac{|\overline{\partial}u(s)|^2}{\Delta \varphi(s)}e^{-2\varphi(s)+2\varphi_{\xi}(s)} dA(s) .
\end{displaymath}
Putting these into \eqref{unique}, using \eqref{dbar-eq}, \eqref{Eq-tau} and the fact that
\[|\overline\partial\chi(s)|^2\lesssim \displaystyle\frac{\chi (s)}{\tau_{\varphi}(s)^2},\]
 we get
\begin{displaymath}
\begin{split}
|u(z)|^2 \omega_{\rho}(z)&\lesssim  \frac{1}{\tau_\rho(z)^2}\int_{\D} \tau_{\varphi} (s)^2\, |f(s)|^2\,\big |\overline{\partial}\chi(s)\big |^2 \,\omega_\rho(s) \, dA(s)
\\
& \lesssim  \frac{1}{\tau_{\varphi}(z)^2}\int_{\D} |f(s)|^2\,\chi(s) \,\omega_\rho(s) \,  dA(s).
\end{split}
\end{displaymath}
Clearly, the function  $\varphi_{\xi}$ is bounded in $D(\delta\tau (\xi))$, and therefore
\[\omega_{\rho} (s)=e^{-2\rho(s)}=e^{-2\varphi(s)+2\varphi_{\xi}(s)}\lesssim e^{-2\varphi(s)}=\omega(s),\qquad s\in D(\delta\tau(\xi)).\]
Thus, we obtain
\begin{displaymath}
\begin{split}
|u(z)|^2 \omega_{\rho}(z)&\lesssim \frac{1}{\tau_{\varphi}(z)^2}\int_{D(\delta\tau(\xi))} |f(s)|^2 \,\omega_\rho(s) \,  dA(s)
\\
&\lesssim  \frac{1}{\tau_{\varphi} (z)^2}\int_{D(\delta\tau(\xi))}{|f(s)|^2 \, \omega (s)\, dA(s)}
\\
&=\frac{1}{\tau_{\varphi} (z)^2}.
\end{split}
\end{displaymath}
This gives,
\[\|K_z\|^2_{_{L^2(\mathbb{D},\chi \omega_\varphi)}}  \lesssim\frac{\omega_{{\rho}}(z)^{-1}}{\tau_{\varphi}(z) ^{2}}. \]
Then, taking into account \eqref{KRN}, we get
\begin{displaymath}
\begin{split}
 |K_z(\xi)|^2 \omega (\xi)
&\lesssim  \frac{1}{\tau_{\varphi}(\xi)^2}\frac{\omega_{\rho}(z)^{-1}}{\tau_{\varphi}(z)^2}
\\
&=\frac{\omega(z)^{-1}}{\left(1+\left(\frac{|\xi-z|}{\beta \,\tau(\xi)}\right)^2\right)^{\frac{M}{2}}}
\frac{1}{\tau_{\varphi}(\xi)^2}\frac{1}{\tau_{\varphi}(z)^2}.
\end{split}
\end{displaymath}
Thus
 \[|K_z(\xi)|\lesssim \beta ^M\,\frac{1}{\tau(\xi)}\,\frac{1}{\tau(z)} \, \omega(\xi)^{-\frac{1}{2}} \,\omega(z)^{-\frac{1}{2}}\,\left(\frac{\tau(\xi)}{|z-\xi|}\right)^{M}.\]
Finally, interchanging the roles of $z$ and $\xi$ we also get
 \[ |K_z(\xi)|\lesssim \beta ^M\,\frac{1}{\tau(\xi)}\,\frac{1}{\tau(z)} \, \omega(\xi)^{-\frac{1}{2}} \,\omega(z)^{-\frac{1}{2}}\,\left(\frac{\tau(z)}{|z-\xi|}\right)^{M},\]
completing the proof of the theorem.
\end{proof}

\begin{lemma}\label{kernelestimat}
Let $\omega\in\mathcal{E}$ and let $K_z $ be the reproducing kernel for $A^2(\omega).$ Then
there exists a constant $ C > 0 $  such that
\[\int_{\D}|K_z(\xi)| \,\omega(\xi)^{\frac{1}{2}}  \,dA(\xi) \leq C \,\omega(z)^{-\frac{1}{2}}. \]
\end{lemma}

\begin{proof}
For  $ 0 < \delta_0 \leq m_\tau$ fixed, let
\[ A(z): =  \int_{|z-\xi|\leq \delta_0\tau(z)} |K_z(\xi)| \,\omega(\xi)^{\frac{1}{2}}\, dA(\xi)\]
and
\[ B(z): =  \int\limits_{|z-\xi| > \delta_0\tau(z)}{|K_z(\xi)|\, \omega(\xi)^{\frac{1}{2}} \, dA(\xi)}.\]
By Lemma \ref{RK-norm} and \eqref{asymptau},
\begin{equation}\label{A}
\begin{split}
A(z)  &\leq  \int_{|z-\xi|\leq \delta_0\tau(z)}{\|K_z\|_{A^2(\omega)}\,\|K_\xi\|_{A^2(\omega)}\,\omega(\xi)^{\frac{1}{2}}\, dA(\xi)}
\\
&\asymp\tau(z)\,\|K_z\|_{A^2(\omega)}
\asymp \omega(z)^{-\frac{1}{2}}.
\end{split}
\end{equation}
On the other hand, by Theorem \ref{RK-PE} with $M=3$, we have
\begin{equation}\label{B}
\begin{split}
B(z)\,&\lesssim \,\,\frac{\omega(z)^{-\frac{1}{2}}}{\tau(z)}\int_{|z-\xi| > \delta_0\tau(z)}\frac{1}{\tau(\xi)}\,\left (\frac{\min(\tau(z),\tau(\xi))}{|z-\xi|}\right )^3 \, dA(\xi)
\\
& \le \,\,\omega(z)^{-\frac{1}{2}} \,\tau(z)\int_{|z-\xi| > \delta_0\tau(z)}\frac{dA(\xi)}{|z-\xi|^3}.
\end{split}
\end{equation}
To estimate the last integral, consider the covering of $\displaystyle\lbrace \xi\in\mathbb{D} : |z-\xi| > \delta_0 \tau(z)\rbrace$ given by
\[R_k(z) = \Big \lbrace\xi\in\D : 2^{k}\delta_0\tau(z) < |z-\xi| \leq 2^{k+1}\delta_0\tau(z) \Big \rbrace, \quad k=0,1,2,\dots\]
We have
\begin{displaymath}
\begin{split}
\int_{|z-\xi| > \delta_0\tau(z)}\frac{dA(\xi)}{|z-\xi|^3}& \,\le \,\sum_{k\ge 0} \int_{R_ k(z)}\frac{dA(\xi)}{|z-\xi|^3}
\\
&\asymp \tau(z)^{-3}\sum_{k\ge 0}2^{-3k} \,Area(R_k(z))
\\
&\asymp \tau(z)^{-1}\sum_{k\ge 0}2^{-k} \lesssim \tau(z)^{-1}.
\end{split}
\end{displaymath}
Putting this into \eqref{B} we get
\begin{equation*}
B(z) \,\lesssim \,\,\omega(z)^{-1/2},
\end{equation*}
which together with \eqref{A} gives the desired  result.
\end{proof}

\section{Bounded projections and the reproducing formula}\label{sec1:2}
Recall that the natural Bergman projection $P_{\omega}$ is given by
\[ P_\omega f(z) = \int_{\D}{f(\xi) \overline{K_z (\xi)} \, \omega(\xi) \, dA(\xi)}.\]
As was said in the introduction, the Bergman projection is not necessarily bounded on $L^p (\D,\omega \,dA)$ unless $p = 2$. However, we are going to see next that $P_{\omega}$ is bounded on $L^p(\omega^{p/2}):=L^p (\D,\omega ^{p/2}\, dA)$.

\begin{theorem}\label{projbd}
Let  $1 \leq p  < \infty$  and  $ \omega\in\mathcal{E}.$ The Bergman projection $ P_\omega : L^p(\omega^{p/2})\longrightarrow A^p(\omega^{p/2})$ is bounded.
\end{theorem}
\begin{proof}
We first consider the easiest case $p= 1.$ By Fubini's Theorem and Lemma \ref{kernelestimat} we obtain
\begin{displaymath}
\begin{split}
\|P_\omega f\|_{A^1(\omega^{1/2})} &=\int_{\D} |P_\omega f(z)| \,\omega(z)^{1/2} \, dA(z)
\\
&\leq \int_{\D}\left (\int_{\D}{|f(\xi)||K_z(\xi})|\,\omega(\xi) \, dA(\xi)\right ) \omega(z)^{1/2}\, dA(z)
\\
&= \int_{\D} |f(\xi)| \,\omega(\xi) \left(\int_{\D}{|K_\xi(z)|\, \omega(z)^{1/2} dA(z)}\right) dA(\xi)
\\
&\lesssim \int_{\D} |f(\xi)| \,\omega(\xi)^{1/2} dA(\xi) =  \|f\|_{L^1(\omega^{1/2})}.
\end{split}
\end{displaymath}
Next, we consider the case $1 <  p < \infty$. Let $ p'$ denote the conjugate exponent of $p$. By H\"older's inequality and Lemma \ref{kernelestimat}
\begin{displaymath}
\begin{split}
|P_{\omega} f(z)|^p &\le \left (\int_{\D} |f(\xi)|^p \, |K_z(\xi)|\, \omega(\xi)^{\frac{p+1}{2}} \, dA(\xi) \right ) \left ( \int_{\mathbb{D}}{|K_z(\xi)| \,\omega(\xi)^{1/2} \ dA(\xi)}\right )^{p-1}
\\
& \lesssim \left (\int_{\D} |f(\xi)|^p \, |K_z(\xi)|\, \omega(\xi)^{\frac{p+1}{2}} \, dA(\xi) \right )\, \omega(z)^{-\frac{(p-1)}{2}}.
\end{split}
\end{displaymath}
This together with Fubini's theorem and another application of Lemma \ref{kernelestimat} gives
\begin{displaymath}
\begin{split}
\|P_\omega f\|_{A^p(\omega^{p/2})}^p &=\int_{\D}{|P_\omega f(z)|^p \omega(z)^{p/2} \, dA(z)}
\\
&\lesssim \int_{\D}\left(\int_{\D} |f(\xi)|^p \, |K_z(\xi)|\, \omega(\xi)^{\frac{p+1}{2}} \, dA(\xi)  \right)\,\omega(z)^{1/2}\, dA(z)
\\
&=\int_{\D}|f(\xi)|^p\, \omega(\xi)^{\frac{p+1}{2}} \left (\int_{\D}|K_{\xi}(z)|\,\omega(z)^{1/2}\, dA(z)\right ) \,dA(\xi)
\\
& \lesssim \|f\|^p_{L^p(\omega^{p/2})}.
\end{split}
\end{displaymath}
The proof is complete.
\end{proof}
To deal with the case $p=\infty$, given a weight $v$, we introduce the growth space $L^{\infty}(v)$ that consists of those measurable functions $f$ on $\D$ such that
\[\|f\|_{L^{\infty}(v)}:=\textrm{ess} \sup_{z\in \D}|f(z)|\,v(z)< \infty, \]
and let $ A^{\infty}(v)$ be the space of all analytic functions in $L^{\infty}(v)$.
\begin{theorem}\label{BP-2}
Let  $ \omega\in\mathcal{E}.$ The Bergman projection
 $ P_\omega : L^{\infty}(\omega^{1/2})\rightarrow A^{\infty}(\omega^{1/2})$ is bounded.
\end{theorem}
\begin{proof}
Let $f\in L^{\infty}(\omega^{1/2})$. By Lemma \ref{kernelestimat}, we get
\begin{displaymath}
\begin{split}
\omega(z)^{1/2}\,|P_\omega f (z)|&\le \,\omega(z)^{1/2}\int_{\D} |f(\xi)|\,|K_  z(\xi)|\,\omega(\xi)\,  dA(\xi)
\\
&\le \,\|f\|_{L^{\infty}(\omega^{1/2})} \cdot \omega(z)^{1/2}\int_{\D} |K_  z(\xi)|\,\omega(\xi)^{1/2}\,  dA(\xi)
\\
&\lesssim  \|f\|_{L^{\infty}(\omega^{1/2})}.
\end{split}
\end{displaymath}
This shows that $P_{\omega}$ is bounded.
\end{proof}
As a consequence of the results obtained on bounded projections, we obtain the following result.
\begin{coro}\label{C-RF}
Let $\omega \in \mathcal{E}$. The following are equivalent:
\begin{enumerate}
\item[(i)] $f=P_{\omega} f$ for every $f\in A^1(\omega ^{1/2})$;

\item[(ii)] $A^2(\omega)$ is dense in $A^1(\omega^{1/2})$.
\end{enumerate}
\end{coro}
\begin{proof}
By the definition of the projection $P_{\omega}$ and the properties of the reproducing kernel, we always have $f=P_{\omega} f$ for every $f\in A^2(\omega)$. Thus (i) is easily implied by the density condition in (ii) and the boundedness of $P_{\omega}$ in $L^1(\omega ^{1/2})$. Conversely, suppose that (i) holds and let $f\in A^1(\omega ^{1/2})$. Since $L^2(\omega)$ is dense in $L^1(\omega ^{1/2})$, we can find functions $g_ n \in L^2(\omega)$ with $\|f-g_ n\|_{L^1(\omega^{1/2})}\rightarrow 0$. Set $f_ n=P_{\omega} g_ n\in A^2(\omega)$. Then, by (i) and Theorem \ref{projbd}, we have
\begin{displaymath}
\|f-f_ n \|_{A^1(\omega^{1/2})}=\|P_{\omega} f -P_{\omega} g_ n\|_{A^1(\omega^{1/2})}\le \|P_{\omega}\|\cdot \|f-g_ n\|_{L^1(\omega^{1/2})}\rightarrow 0,
\end{displaymath}
proving that $A^2(\omega)$ is dense in $A^1(\omega ^{1/2})$. The proof is complete.
\end{proof}
The identity $f=P_{\omega} f$ appearing in (i) is usually referred as the reproducing formula. If the weight $\omega$ is radial, then the polynomials are dense in $A^1(\omega ^{1/2})$ and thus the reproducing formula in (i) holds. The fact that the reproducing formula also holds for non radial weights in the class $\mathcal{E}$ is not obvious, and is our goal to establish that result in the next subsections.

\subsection{Associated weighted Bergman spaces}

We need to consider reproducing kernels $K_ z^*$ of the Bergman space $A^2(\omega_ *)$, where the weight $\omega_{*}$ is of the form
\begin{equation}\label{Eq-AW}
\omega_ *(z)=\omega(z)\,\tau(z)^{\alpha},\qquad \alpha \in \mathbb{R}.
\end{equation}
\begin{lemma}\label{SH-C}
Let $\omega \in \mathcal{L}^*$ and $0<p<\infty$. Then
\[ |f(z)|^p \,\omega_ *(z) \lesssim   \frac{1}{\tau(z)^2}\int_{D(\delta\tau(z))}{|f(\zeta)|^p \,\omega_ *(\zeta) \,   dA(\zeta)},\]
for all $ f \in H(\D)$  and all $\delta > 0 $ sufficiently small.
\end{lemma}
\begin{proof}
This is an immediate consequence of Lemma \ref{subHarmP} and \eqref{asymptau}. Indeed,
\begin{displaymath}
\begin{split}
|f(z)|^p\,\omega_ *(z)&=|f(z)|^p \,\omega(z)\,\tau(z)^{\alpha}\lesssim \tau(z)^{\alpha-2} \int_{D(\delta\tau(z))}|f(\zeta)|^p \,\omega(\zeta) \,   dA(\zeta)
\\
&\asymp \frac{1}{\tau(z)^2} \int_{D(\delta\tau(z))}{|f(\zeta)|^p \,\omega_ * (\zeta) \,   dA(\zeta)}.
\end{split}
\end{displaymath}
This finishes the proof.
\end{proof}

As in Lemma \ref{subHarmP}, it suffices that $f$ be holomorphic in a neighbourhood of $D(\delta \tau(z))$ to get the conclusion in Lemma \ref{SH-C}. As a consequence, we get the estimate $\|K^*_ z\|^2_{A^2(\omega_ *)} \omega_ *(z)\lesssim \tau(z)^{-2}$.

We also need the analogue of Theorem \ref{RK-PE} for the reproducing kernels $K^*_ z$. Since we don't know if $\omega_ *$ belongs to the class $\mathcal{E}$, we can not deduce the result from Theorem \ref{RK-PE}, so that we must repeat the proof with appropriate modifications. Before doing that, we need to establish more estimates of the solutions of the $\overline{\partial}$- equation, a result that can be of independent interest.

\begin{proposition}\label{dbar-C}
Let  $\omega \in \mathcal{E}$ and consider the associated weight
\[\omega_ *(z)=\omega(z)\,\tau(z)^{\alpha},\qquad \alpha \in \mathbb{R}.\]
There is a solution $u$ of the  equation $\overline{\partial} u=f$ satisfying
\[\int_{\D}|u(z)|^p \,\omega_ *(z)^{p/2}\,dA(z)\le C \int_{\D} |f(z)|^p \,\omega_ *(z)^{p/2}\,\tau(z)^p\,dA(z).\]
for $1\le p<\infty$. Moreover, also one has  the $L^{\infty}$-estimate
\begin{displaymath}
\sup_{z\in \D} |u(z)| \,\omega_ *(z)^{1/2}\le C \,\sup_{z\in \D} |f(z)| \,\omega_ *(z)^{1/2}\,\tau(z).
\end{displaymath}
\end{proposition}

\begin{proof}
We follow the method used in \cite{CH}. Fix $\delta \in (0,m_{\tau})$ sufficiently small so that \eqref{RK-Diag} holds, and take $\delta_0, \delta_ 1 >0$ with $2\delta_ 1\le \delta_ 0\le \delta/5$. By Oleinik's covering lemma \cite{O}, there is a sequence of points $\{a_ j\}\subset \D$, such that the following
conditions are satisfied:
\begin{enumerate}
\item[$(i)$] \,$a_ j\notin D(\delta_ 1 \tau(a_ k))$, \,$j\neq k$.

\item[$(ii)$] \, $\bigcup_ j D(\delta_ 1 \tau(a_ j))=\D$.

\item[$(iii)$] \, $\widetilde{D}(\delta _ 1 \tau(a_ j))\subset
D(3\delta_ 1\tau(a_ j))$, where
\[\widetilde{D}(\delta_ 1 \tau(a_
j))=\bigcup_{z\in D(\delta_ 1 \tau(a_ j))}\!\!\!\!D(\delta_ 1 \tau(z)),\quad
j=1,2,\dots\]

\item[$(iv)$] \,$\big \{D(3\delta_ 1 \tau(a_ j))\big \}$ is a
covering of $\D$ of finite multiplicity $N$.\\
\end{enumerate}
Let $\{\chi_ j\}$ be a partition of the unity subordinate to the covering $\{D(\delta_ 1\tau(a_ j))\}.$ For each $a\in \D$, let $h_ a(z)$ denote the normalized reproducing kernels in $A^2(\omega)$, that is, $h_ a(z)=K_ a(z)/\|K_ a\|_{A^2(\omega)}$. By \eqref{RK-Diag}, Lemma \ref{RK-norm} and Theorem \ref{RK-PE}, the functions $h_ a$ satisfy
\begin{enumerate}
\item[(a)] $|h_ a(z)|\asymp \tau(z)^{-1}\,\omega(z)^{-1/2},\qquad z\in D(\delta \tau(a)).$\\

\item[(b)] $\displaystyle {|h_ a(z)|\lesssim \tau(z)^{-1}\,\omega(z)^{-1/2}\left (\frac{\min(\tau(z),\tau(a))}{|z-a|}\right )^M},\qquad z\in \D.$
\end{enumerate}
For each $j$ define
\[ S_ j f(z)=h_{a_ j}(z) \int_{\D} \frac{f(\zeta)\,\chi_ j(\zeta)}{(\zeta-z) h_{a_ j}(\zeta)}\,dA(\zeta).\]
Since $h_{a_ j}(z)$ is holomorphic on $\D$, by the Cauchy-Pompeiu formula we have $\overline{\partial}( S_ j f) =f\chi_ j$, and therefore
\begin{displaymath}
Sf(z)=\sum_ j S_ j f(z)=\omega_ *(z)^{-1/2}\int_{\D} G(z,\zeta) \,f(\zeta)\,\omega_ *(\zeta)^{1/2}\, dA(\zeta)
\end{displaymath}
solves the equation $\overline{\partial}(Sf)=f$, where the kernel $G(z,\zeta)$ is given by
\begin{displaymath}
G(z,\zeta)=\sum_ j \frac{h_{a_ j}(z) \,\chi_ j(\zeta)}{(\zeta-z) h_{a_ j}(\zeta)}\,\,\omega_ *(z)^{1/2}\,\omega_ *(\zeta)^{-1/2}.
\end{displaymath}
Denoting $g(\zeta)=f(\zeta)\,\omega_ *(\zeta)^{1/2}$, the required estimate
\begin{displaymath}
\int_{\D} |Sf(z)|^p \,\omega_ *(z)^{p/2}\,dA(z)\le C \int_{\D} |f(z)|^p\,\omega_ *(z)^{p/2}\,\tau(z)^p\, dA(z)
\end{displaymath}
translates to the estimate
\begin{equation}\label{Req-E}
\int_{\D} |Tg(z)|^p \,dA(z)\le C \int_{\D} |g(z)|^p\,\tau(z)^p\, dA(z),
\end{equation}
with
\[ Tg(z)=\int_{\D} G(z,\zeta)\,g(\zeta)\,dA(\zeta).\]
Now we claim that the integral estimate
\begin{equation}\label{IE-1}
\int_{\D} |G(z,\zeta)| \,\frac{dA(\zeta)}{\tau(\zeta)}\le C
\end{equation}
holds. Indeed,  if $\zeta \in D(\delta_ 0 \tau(z))\cap D(\delta_ 1 \tau(a_ j))$, then using \eqref{asymptau} we see that
\begin{displaymath}
\begin{split}
 |z-a_ j|& \le |z-\zeta|+|\zeta-a_ j| <\delta_ 0 \tau(z)+\delta_ 1 \tau(a_ j)
 \\
 &\le 4\delta_ 0\tau(a_ j)+\delta_ 1 \tau(a_ j) <\delta \tau(a_ j).
\end{split}
\end{displaymath}
 By property (a), $|h_{a_ j}(z)|\asymp \tau(z)^{-1}\,\omega(z)^{-1/2}$, and obviously we also have $|h_{a_ j}(\zeta)|\asymp \tau(\zeta)^{-1}\,\omega(\zeta)^{-1/2}$. Therefore, for $\zeta \in D(\delta_ 0 \tau(z))$,
\begin{displaymath}
|G(z,\zeta)| \lesssim \sum_ j \frac{\chi_ j(\zeta)}{|z-\zeta|}\lesssim \frac{N}{|z-\zeta|},
\end{displaymath}
since there are at most $N$ points $a_ j$ with $\zeta \in D(\delta_ 1\tau(a_ j))$. This gives
\begin{equation}\label{IE-2}
\int_{D(\delta_ 0 \tau(z))} |G(z,\zeta)| \,\frac{dA(\zeta)}{\tau(\zeta)}\le C\frac{1}{\tau(z)}\int_{D(\delta_ 0 \tau(z))} \frac{dA(\zeta)}{|z-\zeta|}\le C,
\end{equation}
after passing to polar coordinates. On the other hand, if $\zeta \notin D(\delta_ 0 \tau(z))$ we use property (b) to get
\begin{displaymath}
\begin{split}
|G(z,\zeta)|&\lesssim \frac{\omega_ *(z)^{1/2}\,\omega_ *(\zeta)^{-1/2} }{\omega(z)^{1/2}\,\tau(z)}
\sum_ j \frac{\chi_ j(\zeta)}{|z-\zeta| \,|h_{a_ j}(\zeta)|}\left (\frac{\min(\tau(z),\tau(a_ j))}{|z-a_ j|}\right )^M
\\
& \asymp \frac{\tau(z)^{\frac{\alpha}{2}-1}\, \omega(\zeta)^{1/2}\,\tau(\zeta)}{\omega_ *(\zeta)^{1/2}}
\sum_ j \frac{\chi_ j(\zeta)}{|z-\zeta| }\left (\frac{\min(\tau(z),\tau(a_ j))}{|z-a_ j|}\right )^M.
\end{split}
\end{displaymath}
Now, for $\zeta \in \big (\D \setminus D(\delta_ 0 \tau(z)) \big ) \cap D(\delta_ 1 \tau(a_ j))$, we have
\begin{displaymath}
\begin{split}
 |z-\zeta|&\le |z-a_ j| +|\zeta-a_ j| \le |z-a_ j| +\delta_ 1 \tau(a_ j)
 \\
 &=|z-a_ j|+\delta_ 1 \tau(z) +\delta_ 1 (\tau(a_ j)-\tau(z))
 \\
 &\le |z-a_ j|+\frac{\delta_ 1}{\delta_ 0} |z-\zeta| +c_ 2\delta_ 1 |z-a_ j|,
 \end{split}
\end{displaymath}
where $c_ 2$ is the constant appearing in condition (B) in the definition of the class $\mathcal{L}$. Since $\delta_ 1/\delta_ 0\le 1/2$, we obtain $|z-\zeta| \le C |z-a_ j|$, which together with the fact that $\tau(\zeta)\asymp \tau(a_ j)$ for $\zeta \in D(\delta_ 1 \tau(a_ j))$ yields
\[\frac{\min(\tau(z),\tau(a_ j))}{|z-a_ j|} \lesssim \frac{\min(\tau(z),\tau(\zeta))}{|z-\zeta|}, \qquad  \zeta \in D(\delta_ 1 \tau(a_ j))\setminus D(\delta_ 0 \tau(z)).\]
Therefore, since there are at most $N$ points $a_ j$ with $\zeta \in D(\delta_ 1\tau(a_ j))$,
\begin{displaymath}
|G(z,\zeta)|\lesssim \frac{\tau(z)^{\frac{\alpha}{2}-1}}{\tau(\zeta)^{\frac{\alpha}{2}-1}\,|z-\zeta|}
\left (\frac{\min(\tau(z),\tau(\zeta))}{|z-\zeta|}\right )^M,\qquad \zeta \notin D(\delta_ 0 \tau(z)).
\end{displaymath}
This gives
\begin{equation*}\label{IE-3}
\begin{split}
\int_{\D\setminus D(\delta_ 0 \tau(z))}\!\!\!\! |G(z,\zeta)| \,\frac{dA(\zeta)}{\tau(\zeta)}
&\lesssim \tau(z)^{\frac{\alpha}{2}-1}\!\int_{\D\setminus D(\delta_ 0 \tau(z))}\!\!\!\!\frac{\big(\min(\tau(z),\tau(\zeta))\big)^M\,dA(\zeta)}{|z-\zeta|^{M+1}\,\tau(\zeta)^{\alpha/2}}
\\
&\le C,
\end{split}
\end{equation*}
where the lat inequality is proved in a similar manner as in the proof of Lemma \ref{kernelestimat}, but in the case $\alpha<0$ one must use that $\tau(\zeta)\lesssim 2^k \tau(z)$ for $\zeta \in R_ k(z)$ (a consequence of condition (B) in the definition of the class $\mathcal{L}$), where $R_ k(z)$ are the same sets used in the proof of Lemma \ref{kernelestimat}. This together with \eqref{IE-2} establish \eqref{IE-1}. \\

Using \eqref{IE-1}, it is straightforward to see that the $L^{\infty}$-estimate
\begin{displaymath}
\sup_{z\in \D} |Sf(z)| \,\omega_ *(z)^{1/2}\le C \sup_{z\in \D} |f(z)| \,\omega_ *(z)^{1/2}\,\tau(z)
\end{displaymath}
holds. Now, let $1\le p<\infty$. By H\"{o}lder's inequality and \eqref{IE-1},
\begin{displaymath}
\begin{split}
|Tg(z)|^p  & \le  \left( \int_{\D} |G(z,\zeta)| \,|g(\zeta)|^p \,\tau(\zeta)^p \,\frac{dA(\zeta)}{\tau(\zeta)} \right) \left ( \int_{\D} |G(z,\zeta)| \,\frac{dA(\zeta)}{\tau(\zeta)}\right )^{p-1}
\\
& \lesssim \int_{\D} |G(z,\zeta)| \,|g(\zeta)|^p \,\tau(\zeta)^p \,\frac{dA(\zeta)}{\tau(\zeta)}.
\end{split}
\end{displaymath}
This, Fubini's theorem and property (a) gives
\begin{displaymath}
\begin{split}
\int_{\D} &|Tg(z)|^p \,dA(z)\lesssim \int_{\D}  |g(\zeta)|^p \,\tau(\zeta)^{p-1} \left (\int_{\D} |G(z,\zeta)| dA(z) \right ) dA(\zeta)
\\
& \le \sum_ j \int_{D(\delta_ 1 \tau(a_ j))} \frac{|g(\zeta)|^p \,\tau(\zeta)^{p-1}}{|h_{a_ j}(\zeta)|\,\omega_ *(\zeta)^{1/2}}
\left (\int_{\D} \frac{|h_{a_ j}(z)| \omega_ *(z)^{1/2}dA(z)}{|z-\zeta|}\right ) dA(\zeta)
\\
& \asymp \sum_ j \int_{D(\delta_ 1 \tau(a_ j))} \frac{|g(\zeta)|^p \,\tau(\zeta)^{p}}{\tau(\zeta)^{\alpha/2}}
\left (\int_{\D} \frac{|h_{a_ j}(z)| \omega_ *(z)^{1/2}dA(z)}{|z-\zeta|}\right ) dA(\zeta).
\end{split}
\end{displaymath}
We handle the inside integral in a similar form as done before. By property (a), for $\zeta \in D(\delta_ 1 \tau(a_ j))$, we have
\begin{equation*}\label{LE-1}
\int_{D(\delta_ 0 \tau(a_ j))} \!\!\frac{|h_{a_ j}(z)| \omega_ *(z)^{1/2}}{|z-\zeta|}dA(z)\asymp \int_{D(\delta_ 0 \tau(a_ j))} \!\!\frac{ \tau (z)^{\frac{\alpha}{2}-1}}{|z-\zeta|}\,dA(z)\lesssim \tau(\zeta)^{\alpha/2}.
\end{equation*}
The integral from outside the disc $D(\delta_ 0 \tau(a_ j))$ is estimated with the same method as done in the proof of \eqref{IE-1} using property (b), so that we obtain
\begin{displaymath}
\int_{\D} \frac{|h_{a_ j}(z)| \omega_ *(z)^{1/2} dA(z)}{|z-\zeta|} \lesssim \tau(\zeta)^{\alpha/2}.
\end{displaymath}
Putting this in the previous estimate we finally obtain
\begin{displaymath}
\int_{\D} |Tg(z)|^p \,dA(z)\lesssim \sum_ j \int_{D(\delta_ 1 \tau(a_ j))} \!\!\!\! \!\!|g(\zeta)|^p \,\tau(\zeta)^{p}
 dA(\zeta)\lesssim \int_{\D} |g(\zeta)|^p \,\tau(\zeta)^{p}
 dA(\zeta),
\end{displaymath}
since $\{D(\delta_ 1 \tau(a_ j))\}$ is a covering of $\D$ of finite multiplicity. This proves \eqref{Req-E} completing the proof of the proposition.
\end{proof}

Now we can prove the analogue of Theorem \ref{RK-PE} for the the reproducing kernels $K^*_ z$.

\begin{lemma}\label{PE-C}
Let $\omega \in \mathcal{E}$, and $K^*_z $ be the reproducing kernel of  $A^2(\omega_ *)$ where $\omega_ *$ is the associated weight given by \eqref{Eq-AW}. For each $M\ge 1$, there exists a constant $ C > 0$ (depending on $M$) such that for each  $z,\xi \in \D$ one has
\begin{equation*}
 |K^*_z(\xi)|\leq C \frac{1}{\tau(z)}\frac{1}{\tau(\xi)}\,\omega_ *(z)^{-1/2}\omega_ *(\xi)^{-1/2} \left (\frac{\min(\tau(z),\tau(\xi))}{|z-\xi|}\right )^M.
\end{equation*}
\end{lemma}

\begin{proof}
Let $ z,\xi \in\mathbb{D}$ and fix $0<\delta<m_{\tau}$. The result is clear if $D(\delta\tau(z))\cap D(\delta\tau(\xi)) \neq \emptyset ,$ so that we assume $D(\delta\tau(z))\cap D(\delta\tau(\xi)) = \emptyset. $
Let $0 \leq \chi \leq 1 $ be a function in $C^\infty (\D)$ with compact support in the disk $  D(\delta\tau(\xi))$  such that  $\chi \equiv 1$  in   $D(\frac{\delta}{2}\, \tau(\xi))$
and  $|\overline\partial\chi|^2\lesssim \displaystyle\frac{\chi}{\tau(\xi)^2}.$  By Lemma \ref{SH-C} we obtain
\begin{equation}\label{KRN-A}
\begin{split}
|K^*_z(\xi)|^2 \omega_ *(\xi)&\lesssim\frac{1}{\tau(\xi)^2}\,\big \|K^*_z \big \|^2_{L^2(\D,\,\chi \omega_ * dA)}.
\end{split}
\end{equation}
 By duality,
$\|K^*_z\|_{_{L^2(\D,\chi\omega_ * )}} = \sup_{\substack{f}}|\langle f,K^*_z\rangle_{L^2(\D,\chi\omega_ * dA)}|,$  where the supremum runs over all holomorphic functions $f$ on $D(\delta\tau(\xi))$ such that
\begin{equation}\label{Eq-f}
 \int_{D(\delta\tau(\xi))}{|f(z)|^2\,\omega_ *(z)\,dA(z)}= 1.
\end{equation}
As  $ f\chi\in L^2(\D,\omega_ *\,  dA)$ one has $ \langle f,K^*_z\rangle_{L^2(\D,\chi\omega_ * dA)} = P_{\omega_ *}(f\chi)(z),$
where $P_{\omega_ *}$ is the orthogonal Bergman projection
 which is obviously bounded from $L^2(\D,\omega_ *\, dA)$   to  $A^2(\omega_ *).$ Now we consider $ u = f\chi - P_{\omega_ *}(f\chi)$ the solution with minimal norm in  $L^2(\D,\omega_ *\, dA)$ of the equation
\begin{equation}\label{dbar-eqA}
\overline\partial{u} = \overline\partial ({f\chi}) = f\overline\partial\chi.
\end{equation}
Since $\chi(z) = 0,$  we get
 $|\langle f,K^*_z\rangle_{L^2(\D,\chi\omega_ *)}|=| P_{\omega_ *}(f\chi)(z)| = |u(z)|.$
For a given $0<\varepsilon <1/2$,  consider the subharmonic function from Lemma \ref{Lfi} given by
\begin{displaymath}
 \varphi_ {\xi}(s) = \frac{M}{4} \log \left (1+\frac{|s-\xi|^2}{\beta ^2 \tau(\xi)^2} \right ),
\end{displaymath}
with $\beta>0$ (depending on $M$ and $\varepsilon$) taken big enough so that
\begin{equation}\label{DELTAA3}
\big |\overline{\partial} \varphi_ {\xi} (s)\big |^2 \le \varepsilon \,\Delta \varphi(s),\quad \textrm{ and }\quad \Delta \varphi_ {\xi}(s)\le \varepsilon \, \Delta \varphi (s).
\end{equation}
 Thus  $\Delta \varphi \le \Delta (\varphi+\varphi_{\xi}) \le 2 \Delta \varphi$ and $\frac{1}{2}\,\Delta \varphi \le \Delta (\varphi-\varphi_{\xi}) \le  \Delta \varphi$.
Next, we are going to apply the method used in the proof of Berndtsson's theorem. Since $u$ is the solution with minimal norm in  $L^2(\D,\omega_ *\, dA)$ of the equation \eqref{dbar-eqA}, it satisfies $\langle u,h \rangle _{\omega_ *}=0$
for any square integrable holomorphic function $h$ in $\D$. This clearly implies that
\[\int_{\D} u_ 0 \,\overline{h} \,\omega_ *\,e^{-2\varphi_{\xi}} dA=0\]
for any such $h$, with $u_ 0 =u \,e^{2\varphi_{\xi}}$. Thus $u_ 0$ is the minimal solution in $L^2(\D,\omega_ *\,e^{-2\varphi_{\xi}} dA)$ of the equation $\overline{\partial}v=\overline{\partial}(u \,e^{2\varphi_{\xi}}):=g$.
By Proposition \ref{dbar-C} applied with the weight $\omega _{\xi}=\omega e^{-2\varphi_{\xi}}$, we can find a solution $v$ of the equation $\overline{\partial}v=g$ satisfying
\[ \int_{\D}|v|^2\,\omega_ *\, e^{-2\varphi_{\xi}}\, dA \leq  C \int_{\D}|g|^2\, \omega_ *\, e^{-2\varphi_{\xi} }\,\tau^2\,dA .\]
Hence the same estimate is true for the minimal solution $u_ 0$, which implies
\begin{equation*}
\begin{split}
\int_{\D} |u|^2 \,\omega_ *\, e^{2\varphi_{\xi}} dA &\le C \int_{\D}\big |\overline{\partial} u+u \,\overline{\partial}\varphi_{\xi}\big |^2\,\omega_ *\, e^{2\varphi_{\xi}}\,\tau^2\,  dA
\\
&\le C \int_{\D}\big |\overline{\partial} u\big |^2\,\omega_ *\, e^{2\varphi_{\xi}}\,\tau^2\,  dA
+ C \int_{\D}\big |u\,\overline{\partial}\varphi_{\xi}\big |^2\,\omega_ *\, e^{2\varphi_{\xi}}\,\tau^2\,  dA.
\end{split}
\end{equation*}
Now use \eqref{DELTAA3} with $\varepsilon >0$ taken so that $C\varepsilon \le 1/2$,  and absorb the last member of the right hand side in the left hand side. The result is
\begin{equation}\label{Eq-BT1}
\int_{\D} |u|^2 \,\omega_ *\, e^{2\varphi_{\xi}} dA \le C \int_{\D}\big |\overline{\partial} u\big |^2\,\omega_ *\, e^{2\varphi_{\xi}}\,\tau^2\,  dA.
\end{equation}
Arguing as in the proof of Lemma \ref{SH-C}, then applying \eqref{Eq-BT1}, we obtain
\begin{equation*}
\begin{split}
|u(z)|^2 \omega_{*}(z)\,e^{2\varphi_{\xi}(z)} &\lesssim  \frac{1}{\tau(z)^2}\int_{\D}|u(s)|^2 \,\omega_* (s) \,e^{2\varphi_{\xi}(s)}\,dA(s)
\\
& \lesssim  \frac{1}{\tau(z)^2}\int_{\D}|f(s)|^2 \,|\overline\partial\chi(s)|^2\,\omega_* (s) \,e^{2\varphi_{\xi}(s)}\,\tau(s)^2\,dA(s)
\\
& \lesssim  \frac{1}{\tau(z)^2}\int_{D(\delta \tau(\xi))}|f(s)|^2 \,\omega_* (s) \,e^{2\varphi_{\xi}(s)}\,dA(s).
\end{split}
\end{equation*}
Since the function  $\varphi_{\xi}$ is bounded in $D(\delta\tau (\xi))$, this and \eqref{Eq-f} yields
\[ |u(z)|^2 \omega_{*}(z)\,e^{2\varphi_{\xi}(z)} \tau(z)^2 \le C.\]
Thus, taking into account \eqref{KRN-A}, we get
 \[|K^*_z(\xi)|\lesssim \beta ^M\,\frac{1}{\tau(\xi)}\,\frac{1}{\tau(z)} \, \omega_ *(\xi)^{-1/2} \,\omega_ *(z)^{-1/2}\,\left(\frac{\tau(\xi)}{|z-\xi|}\right)^{M}.\]
Finally, interchanging the roles of $z$ and $\xi$ we  obtain the desired result.
\end{proof}

\begin{coro}\label{RKE-A}
Let $\omega\in\mathcal{E}$, and  $K^*_z $ be be the reproducing kernel for $A^2(\omega_ *),$ where $\omega_ *$ is the associated weight given by \eqref{Eq-AW}. For  $\beta\in \mathbb{R}$,
there exists a constant $ C > 0 $  such that
\[\int_{\D}{|K^*_z(\xi)|\, \omega_ *(\xi)^{1/2}\,\tau(\xi)^{\beta}  \,dA(\xi)} \leq C \,\omega_ *(z)^{-1/2}\,\tau(z)^{\beta}. \]
\end{coro}

\begin{proof}
Apart from the extra factor $\tau(z)^{\beta}$, this is almost the analogue of Lemma \ref{kernelestimat}. For the proof, just use the same method applying Lemma \ref{PE-C} with $M$ taken big enough, but in the case $\beta-1>0$, use that $\tau(\xi)\lesssim 2^k \tau(z)$ for $z\in R_ k(z)$.
\end{proof}

Arguing in the same way as in the proof of the boundedness of the Bergman projection, using Corollary \ref{RKE-A} with $\beta=0$, we can prove that $P_{\omega_ *}$ is bounded on $L^p(\omega_ *^{p/2})$, but in order to obtain the reproducing formula, what is really needed is the following result.

\begin{lemma}\label{BP-A}
Let $\omega \in \mathcal{E}$, $1\le p< \infty$ and let $\omega_ *$ be the associated weight  given by \eqref{Eq-AW}. Then $P_{\omega_ *}:L^p(\omega^{p/2})\rightarrow A^p(\omega^{p/2})$ is bounded.
\end{lemma}
\begin{proof}
This is proved with the same method used in the proof of Theorem \ref{projbd}, but using Corollary \ref{RKE-A} instead of Lemma \ref{kernelestimat}. We left the details to the interested reader.
\end{proof}

\subsection{The reproducing formula}

With all the machinery developed in the previous subsection, we can prove the following key result from which the reproducing formula will follow.

\begin{lemma}\label{Lem-D}
Let $\omega \in \mathcal{E}$ and $f\in A^1(\omega ^{1/2})$. Then one can find
 functions $f_ n\in A^2(\omega)$ with $\|f_ n\|_{A^1(\omega ^{1/2})}\lesssim \|f\|_{A^1(\omega ^{1/2})}$ such that $f_ n \rightarrow f$ uniformly on compact subsets of $\D$.
\end{lemma}

\begin{proof}
Our proof has his roots in an argument used by Lindh\"{o}lm  \cite{Lind} in the setting of weighted Fock spaces. Let $r_ n:=1-1/n$, and consider a sequence of $C^{\infty}$ functions $\chi_ n$ with compact support on $\D$ such that $\chi_ n(z)=1$ for $|z|\le 1-1/n$, and $|\overline{\partial}\chi_ n |\lesssim \, n.$
For each $n$, consider the analytic functions
\[f_ n =P_{\omega_ *}(f\chi_ n),\]
where $\omega_ *$ is the associated weight given by
\[\omega_ *(z)=\omega(z)\,\tau(z)^{2}.\]
Since the functions $f\chi_ n \in L^2(\omega)$ and $P_{\omega _ *}$ is bounded on $L^p(\omega ^{p/2})$, $1\le p<\infty$, then $f_ n \in A^2(\omega)$, and
\[\|f_ n\|_{A^1(\omega ^{1/2})}=\|P_{\omega_ *}(f\chi_ n)\|_{A^1(\omega ^{1/2})}\lesssim \|f\chi_ n\|_{L^1(\omega ^{1/2})}\le \|f\|_{A^1(\omega ^{1/2})}.\]
Therefore, it remains to show that $f_ n\rightarrow f$ uniformly on compact subsets of $\D$. Since
$
|f-f_ n|\le |f-f\chi _ n |+|f\chi_ n-f_ n|,
$
and, clearly, $f\chi _ n \rightarrow f$ uniformly on compact subsets of $\D$, it is enough to show that $u_ n\rightarrow 0$ uniformly on compact subsets of $\D$, with $u_ n=f\chi_ n-P_{\omega_ *}(f\chi_ n)$.

Fix $0<R<1$ and let $z\in \D$ with $|z|\le R$. For $n$ big enough, the function $u_ n$ is analytic in a neighborhood of the disc $D(\delta_ 0 \tau(z))$, with $\delta_ 0\in (0,m_{\tau})$. Hence, by Lemma \ref{subHarmP},
\begin{equation}\label{Eq-u0}
\begin{split}
\tau(z)^4 \,|u_ n(z)|^2\,e^{-2\varphi(z)} & \lesssim \tau(z)^2 \int_{D(\delta_ 0 \tau(z))} |u_ n(\zeta)|^2\,e^{-2\varphi(\zeta)} dA(\zeta)
\\
& \lesssim  \int_{D(\delta_ 0 \tau(z))} |u_ n(\zeta)|^2\,e^{-2\varphi(\zeta)} \,\tau(\zeta)^2\, dA(\zeta)
\\
& \le   \int_{\D} |u_ n(\zeta)|^2\,\omega_ *(\zeta)\, dA(\zeta)
\end{split}
\end{equation}

Since $u_ n$ is the solution of the $\overline{\partial}$-equation $\overline{\partial}v=f\overline{\partial}\chi_ n$ with minimal $L^2(\omega_ *)$ norm, by Proposition \ref{dbar-C}, we have
\begin{displaymath}
\int_{\D} |u_ n|^2\,\omega_ *\, dA \le C \int_{\D}  |f\overline{\partial}\chi_ n|^2 \,\omega_ *\,\tau ^2\, dA.
\end{displaymath}
Since  $\overline{\partial}\chi_ n$ is supported on $r_ n<|z|<1$ with $|\overline{\partial}\chi_ n|\lesssim n$, we get
\begin{displaymath}
\begin{split}
\int_{\D} |u_ n|^2\,\omega_ * \, dA
&\le C n^2 \int_{|\zeta|> r_ n}  |f(\zeta)|^2 \,\omega(\zeta)\,\tau(\zeta)^4 dA(z).
\end{split}
\end{displaymath}
Since $\tau(\zeta)\lesssim (1-|\zeta|)\le 1/n$ for $|\zeta|>r_ n$, using the pointwise estimate (a consequence of Lemma \ref{subHarmP})
\[|f(\zeta)|\,\omega(\zeta)^{1/2}\tau(\zeta)^2 \lesssim \|f\|_{A^1(\omega ^{1/2})},\]
we obtain
\begin{displaymath}
\begin{split}
\int_{\D} |u_ n|^2\,\omega_ * \,dA & \le C  \int_{|\zeta|> r_ n}  |f(\zeta)|^2 \,\omega(\zeta)\,\tau(\zeta)^2 dA(z)
\\
&\le C  \|f\|_{A^1(\omega ^{1/2})} \int_{|\zeta|> r_ n}  |f(\zeta)| \,\omega(\zeta)^{1/2}\, dA(z),
\end{split}
\end{displaymath}
and this goes to zero as $n\rightarrow \infty$ since $f\in A^1(\omega ^{1/2})$. Bearing in mind \eqref{Eq-u0}, this implies that $u_ n\rightarrow 0$ uniformly on compact subsets of $\D$, finishing the proof.
\end{proof}

Now we are ready to show that, for weights $\omega\in \mathcal{E}$, the reproducing formula $f=P_{\omega} f$ holds for any $f\in A^1(\omega ^{1/2})$. In view of Corollary \ref{C-RF} this would give the density of $A^2(\omega)$ in $A^1(\omega ^{1/2})$.

\begin{theorem}\label{RF-v1}
Let $\omega \in \mathcal{E}$. Then $f=P_{\omega}f$ for each $f\in A^1(\omega ^{1/2})$.
\end{theorem}

\begin{proof}
Let $f\in A^1(\omega ^{1/2})$. By Lemma \ref{Lem-D} one can find
 functions $f_ n\in A^2(\omega)$ with $\|f_ n\|_{A^1(\omega ^{1/2})}\lesssim \|f\|_{A^1(\omega ^{1/2})}$ such that $f_ n \rightarrow f$ uniformly on compact subsets of $\D$. Then
\begin{displaymath}
|f(z)-P_{\omega} f(z)|\le |f(z)-f_ n(z)|+|f_ n(z)-P_{\omega} f(z)|.
\end{displaymath}
Clearly, the first term goes to zero as $n\rightarrow \infty$. For the second term, since $f_ n\in A^2(\omega)$, one has the reproducing formula $f_ n=P_{\omega}f_ n$, and therefore
\begin{displaymath}
|f_ n(z)-P_{\omega}f(z)|=|P_{\omega}(f_ n-f)(z)|\le \int_{\D} |f_ n(\xi)-f(\xi)|\,|K_ z(\xi)|\,\omega(\xi)\,dA(\xi).
\end{displaymath}
Fix $0<\delta<m_{\tau}$ and split the previous integral in two parts: one integrating over the disk $D(\delta \tau(z))$, and the other over $\D\setminus D(\delta \tau(z))$. We have
\begin{displaymath}
\begin{split}
 \int_{D(\delta \tau(z))} |f_ n(\xi)-f(\xi)|\,&|K_ z(\xi)|\,\omega(\xi)\,dA(\xi)
 \\
 &\lesssim \frac{\|K_ z\|_{A^2(\omega)}}{\tau(z)} \int_{D(\delta \tau(z))}\!\! |f_ n(\xi)-f(\xi)|\,\omega(\xi)^{1/2}\,dA(\xi)
\end{split}
\end{displaymath}
and this goes to zero as $n\rightarrow \infty$ since $\overline{D(\delta \tau(z))}\subset \D$ and $f_ n\rightarrow f$ uniformly on compact subsets of $\D$. On the other hand, if $\xi \notin D(\delta \tau(z))$, then we apply the estimate for the reproducing kernel obtained in Theorem \ref{RK-PE} with $M=3$ to get
\begin{displaymath}
|K_ z(\xi)| \lesssim \frac{\|K_ z\|_{A^2(\omega)} }{\omega(\xi)^{1/2}\,\tau(\xi)}\,\left (\frac{\min(\tau(z),\tau(\xi))}{|z-\xi|}\right )^M
\lesssim \frac{\|K_ z\|_{A^2(\omega)} }{\omega(\xi)^{1/2}}\,\frac{\tau(\xi)^2}{\tau(z)^3}.
\end{displaymath}
Therefore, we obtain
\begin{displaymath}
\begin{split}
 \int_{\D\setminus D(\delta \tau(z))} &|f_ n(\xi)-f(\xi)|\,|K_ z(\xi)|\,\omega(\xi)\,dA(\xi)
 \\
 &\lesssim \frac{\|K_ z\|_{A^2(\omega)}}{\tau(z)^3} \int_{\D} |f_ n(\xi)-f(\xi)|\,\omega(\xi)^{1/2}\,\tau(\xi)^2\, dA(\xi)=I_{1,n}+I_{2,n},
 \end{split}
\end{displaymath}
with
\begin{displaymath}
 I_{1,n}= \frac{\|K_ z\|_{A^2(\omega)}}{\tau(z)^3} \int_{|\xi|\le R} |f_ n(\xi)-f(\xi)|\,\omega(\xi)^{1/2}\,\tau(\xi)^2\, dA(\xi)
\end{displaymath}
and
\begin{displaymath}
I_{2,n}= \frac{\|K_ z\|_{A^2(\omega)}}{\tau(z)^3} \int_{R<|\xi|<1} |f_ n(\xi)-f(\xi)|\,\omega(\xi)^{1/2}\,\tau(\xi)^2\, dA(\xi).
\end{displaymath}
By Lemma \ref{subHarmP} and $\|f_ n\|_{A^1(\omega ^{1/2})}\lesssim \|f\|_{A^1(\omega ^{1/2})}$ it follows that
\[|f_ n(\xi)-f(\xi)|\,\omega(\xi)^{1/2}\,\tau(\xi)^2\lesssim \|f\|_{A^1(\omega ^{1/2})},\]
 and therefore
\[I_{2,n} \lesssim \frac{\|K_ z\|_{A^2(\omega)}}{\tau(z)^3}\,\|f\|_{A^1(\omega ^{1/2})} \int_{R<|\xi|<1} \!\!dA(\xi).\]
By taking $0<R<1$ enough close to $1$ we can make the last expression as small as desired. Once $R$ is taken, then $I_{1,n}\rightarrow 0$ since $f_ n$ converges to $f$ uniformly on compact subsets of $\D$. This shows that $f(z)=P_{\omega} f(z)$.
\end{proof}

\section{Complex interpolation}

An elementary introduction to the basic theory of complex interpolation, including the complex interpolation of $L^p$ spaces can be found in Chapter 2 of the book \cite{Zhu}. We assume in this section that the reader is familiar with that theory. First of all, we recall the following well-known interpolation theorem of Stein and Weiss \cite{SW}.

\begin{otherth}\label{T-SW}
Suppose that $\omega, \omega _ 0$ and $\omega _ 1$ are weight functions on $\D$. If $1\le p_ 0\le p_ 1\le \infty$ and $0\le \theta \le 1$, then
\[\big [ L^{p_ 0}(\D,\omega _ 0 dA), L^{p_ 1}(\D,\omega _ 1 dA) \big ]_{\theta}=L^p(\D,\omega dA)\]
with equal norms, where
\[\frac{1}{p}=\frac{1-\theta}{p_ 0}+\frac{\theta}{p_ 1},\qquad \omega ^{1/p}=\omega _ 0 ^{\frac{1-\theta}{p_ 0}}\,\omega _ 1 ^{\frac{\theta}{p_ 1}}.\]
\end{otherth}

With this and the result on bounded projections we can obtain the following result on complex interpolation of large weighted Bergman spaces.

\begin{theorem}
Let $\omega$ be a weight in the class $\mathcal{E}$.  If $1\le p_ 0\le p_ 1\le \infty$ and $0\le \theta \le 1$, then
\[[ A^{p_ 0}(\omega^{p_ 0/2}), A^{p_ 1}(\omega^{p_ 1/2})]=A^p(\omega ^{p/2}),\]
where
\[\frac{1}{p}=\frac{1-\theta}{p_ 0}+\frac{\theta}{p_ 1}.\]
\end{theorem}

\begin{proof}
The inclusion $[ A^{p_ 0}(\omega^{p_ 0/2}), A^{p_ 1}(\omega^{p_ 1/2})]\subset A^p(\omega ^{p/2})$ is a consequence of the definition of complex interpolation, the fact that each $A^{p_ k}(\omega ^{p_ k/2})$ is a closed subspace of $L^{p_ k}(\omega ^{p_ k/2})$ and $\big [ L^{p_ 0}(\omega ^{p_ 0/2}), L^{p_ 1}(\omega ^{p_ 1/2} ) \big ]_{\theta}=L^p(\omega ^{p/2})$. This last assertion follows from Theorem \ref{T-SW}.

On the other hand, if $f\in A^p(\omega ^{p/2})\subset L^p(\omega ^{p/2})$, it follows from Theorem \ref{T-SW} that
\[\big [ L^{p_ 0}(\omega ^{p_ 0/2}), L^{p_ 1}(\omega ^{p_ 1/2}) \big ]_{\theta}=L^p(\omega ^{p/2}).\]
Thus, there exists a function $F_{\zeta}(z)$ ($z\in \D$ and $0\le \Re \,\zeta \le 1$) and a positive constant $C$ such that:
\begin{itemize}
\item[(a)] $F_{\theta}(z)=f(z)$ for all $z\in \D$.

\item[(b)] $\|F_{\zeta}\|_{L^{p_ 0}(\omega ^{p_ 0/2})}\le C$ for all $\Re \,\zeta =0$.

\item[(c)] $\|F_{\zeta}\|_{L^{p_ 1}(\omega ^{p_ 1/2})}\le C$ for all $\Re \,\zeta =1$.
\end{itemize}
Define a function $G_{\zeta}$ by
$G_{\zeta}(z)=P_{\omega} F_{\zeta}(z).$
Due to the reproducing formula in Theorem \ref{RF-v1}, and the boundedness of the Bergman projection, see Theorem \ref{projbd}, we have:
\begin{itemize}
\item[(a)] $G_{\theta}(z)=f(z)$ for all $z\in \D$.

\item[(b)] $\|G_{\zeta}\|_{L^{p_ 0}(\omega ^{p_ 0/2})}\le C$ for all $\Re \,\zeta =0$.

\item[(c)] $\|G_{\zeta}\|_{L^{p_ 1}(\omega ^{p_ 1/2})}\le C$ for all $\Re \,\zeta =1$.
\end{itemize}
Since each function $G_{\zeta}$ is analytic on $\D$, we conclude that $f$ belongs to $ [ A^{p_ 0}(\omega^{p_ 0/2}), A^{p_ 1}(\omega^{p_ 1/2})]$. This completes the proof of the theorem.
\end{proof}

\section{Duality}\label{sec2:2}

As in the case of the standard Bergman spaces, one can use the result just proved on the boundedness of the Bergman projection $P_{\omega}$ in $L^p(\omega ^{p/2})$ to identify the dual space of $A^p (\omega ^{p/2})$. As usual, if $X$ is a normed space, we denote its dual by $X^*$. For a given weight $v$, we introduce the space $A_ 0(v)$ consisting of those functions $f\in A^{\infty}(v)$ with $\lim _{|z|\rightarrow 1^-} v(z) |f(z)|=0$. Clearly, $A_ 0(v)$ is a closed subspace of $A^{\infty}(v)$.
\begin{lemma}\label{LA0}
Let $\omega \in \mathcal{E}$ and $z\in \D$. Then $K_ z\in A_ 0(\omega ^{1/2})$.
\end{lemma}
\begin{proof}
Let $\delta \in (0,m_{\tau})$. If $|\xi|$ is enough close to $1$, then $\xi \notin D(\delta \tau(z))$, and it follows from the pointwise estimate in Theorem \ref{RK-PE} (taking $M=2$) that
\begin{displaymath}
\omega (\xi)^{1/2} |K_ z(\xi)| \lesssim \|K_ z \|_{A^2(\omega)} \frac{\tau(\xi)}{\tau(z)^2}\rightarrow 0
\end{displaymath}
as $|\xi|\rightarrow 1^{-}$ since $\tau(\xi)\lesssim (1-|\xi|)$. This finishes the proof.
\end{proof}
In particular, since $A_ 0(\omega ^{1/2})\subset A^{\infty}(\omega ^{1/2})\subset A^p(\omega ^{p/2})$, it follows that $K_ z\in A^p(\omega ^{p/2})$ for any $p$. Now we are ready to state and prove the corresponding duality results.
\begin{theorem}\label{Dual-Ber}
Let $\omega \in \mathcal{E}$ and $1<p<\infty$.  The dual space of $A^p(\omega ^{p/2})$ can be identified (with equivalent norms) with $A^{p'}(\omega ^{p'/2})$ under the integral pairing
\[\langle f, g \rangle _{\omega}=\int_{\D} f(z)\,\overline{g(z)}\,\omega (z)\,dA(z).\]
Here $p'$ denotes the conjugate exponent of $p$, that is, $p'=p/(p-1)$.
\end{theorem}
\begin{proof}
Given a function $g\in A^{p'}(\omega^{p'/2}),$  H\"older's inequality implies that the linear functional $\Lambda_ g:A^p(\omega^{p/2}) \rightarrow \mathbb{C}$ given by
\begin{displaymath}
 \Lambda_{g}(f)= \int_{\D} f(z)\,\overline{g(z)}\,\omega (z)\,dA(z)
\end{displaymath}
is bounded with $\|\Lambda _ g \|\le \|g\|_{A^{p'}(\omega ^{p'/2})}$.
Conversely, let $\Lambda \in \big (A^p(\omega^{p/2})\big )^*$. By Hahn-Banach, we can  extend $\Lambda$ to an element $\widetilde{\Lambda}\in ( L^p(\omega^{p/2}))^*$ with $\|\widetilde{\Lambda}\|= \|\Lambda\|.$ By the $L^p$ Riesz representation theorem there exists $H \in L^{p'}(\omega^{p/2})$ with $\|H\|_{L^{p'}(\omega^{p/2})}=\|\widetilde{\Lambda}\|= \|\Lambda\|$ such that
\[\widetilde{\Lambda} (f)= \int_{\D}f(z)\,\overline{H(z)}\,\omega(z)^{p/2}\, dA(z),\]
for every $f\in A^p(\omega^{p/2}).$ Consider the function $h(z)=H(z)\,\omega(z)^{p/2-1}$. Then $h\in L^{p'}(\omega^{p'/2})$ with
\[\|h\|_{L^{p'}(\omega^{p'/2})}=\|H\|_{L^{p'}(\omega^{p/2})}=\|\Lambda\|,\]
and
\[\Lambda (f)=\widetilde{\Lambda} (f)= \int_{\D}f(z)\,\overline{h(z)}\,\omega(z)\,dA(z),\qquad f\in A^p(\omega^{p/2}).\]
Let $g = P_{\omega}h.$ By Theorem \ref{projbd}, $g\in A^{p'}(\omega^{p'/2})$ with
 $$\|g\|_{A^{p'}(\omega^{p'/2})}=\|P_{\omega} h\|_{A^{p'}(\omega^{p'/2})}\lesssim \|h\|_{L^{p'}(\omega^{p'/2})}=\|\Lambda\|.$$
From Fubini's theorem  it is easy to see that $P_{\omega}$ is self-adjoint. By Theorem \ref{RF-v1}, the reproducing formula $f=P_{\omega}f$ holds for every $f\in  A^{p}(\omega^{p/2})\subset A^1(\omega ^{1/2})$. Then, one has
\begin{displaymath}
\begin{split}
\Lambda(f)&=\int_{\D} f(z) \,\overline{h(z)}\, \omega(z)\, dA(z)=\langle P_{\omega} f,h \rangle_{\omega}
\\
&=\langle f,P_{\omega} h \rangle_{\omega}= \int_{\D} f(z) \,\overline{g(z)}\, \omega(z)\, dA(z)=\Lambda_ g (f).
\end{split}
\end{displaymath}
Finally, the function $g$ is unique. Indeed, if there is another function $\widetilde{g}\in A^{p'}(\omega ^{p'/2})$ with $\Lambda (f)=\Lambda_ g(f)=\Lambda_{\widetilde{g}}(f)$ for $f\in A^p(\omega ^{p/2})$, then by testing the previous identity on the reproducing kernels $K_ a$ for each $a\in \D$, and using the reproducing formula, we obtain
$$g(a)=\overline{\Lambda_ g (K_ a)}=\overline{\Lambda_ {\widetilde{g}} (K_ a)}=\widetilde{g} (a),\qquad a\in \D.$$
Thus, any bounded linear functional $\Lambda$ on $A^p(\omega^{p/2})$ is of the form $\Lambda=\Lambda_ g$ for some unique $g\in A^{p'}(\omega^{p'/2})$ and, furthermore,
$$\|\Lambda \|\asymp \|g\|_{A^{p'}(\omega^{p'/2})}.$$
The proof is complete.
\end{proof}

\begin{theorem}\label{dual1}
Let $\omega\in \mathcal{E}$. The dual space of $A^1(\omega ^{1/2})$ can be identified (with equivalent norms) with $A^{\infty}(\omega ^{1/2})$ under the integral pairing $\langle f,g \rangle _{\omega}$.
\end{theorem}
\begin{proof}
If  $g\in A^{\infty}(\omega^{1/2}),$ clearly the linear functional $\Lambda_ g:A^1(\omega^{1/2}) \rightarrow \mathbb{C}$ given by
$\Lambda_{g}(f)= \langle f,g\rangle _{\omega}$
is bounded with $\|\Lambda _ g \|\le \|g\|_{A^{\infty}(\omega ^{1/2})}$.

Conversely, let $\Lambda \in \big (A^1(\omega^{1/2})\big )^*$.
In particular, $\Lambda$ is a bounded linear functional in $A^2(\omega)$, and hence, there exists a unique function $g\in A^2(\omega)$ with $\Lambda (f)=\langle f,g \rangle _{\omega}$ whenever $f$ is in $A^2(\omega)$. Then $g(z)=\langle g,K_ z \rangle _{\omega}=\overline{\Lambda (K_ z)}$, and by Lemma \ref{kernelestimat} we get
$$|g(z)|=|\Lambda (K_ z)\|\le \|\Lambda\|\cdot \|K_ z\|_{A^1(\omega ^{1/2})}\lesssim \|\Lambda \| \,\omega (z)^{-1/2},$$
which shows that $g$ actually belongs to $A^{\infty}(\omega ^{1/2})$. Finally, by Theorem \ref{RF-v1} and Corollary \ref{C-RF}, $A^2(\omega)$ dense in $A^1(\omega ^{1/2})$ and therefore we have $\Lambda (f)=\langle f,g \rangle _{\omega}$ for all functions $f$ in $A^1(\omega ^{1/2})$.
\end{proof}

\begin{coro}
Let $\omega\in \mathcal{E}$. The set $E$ of finite linear combinations of reproducing kernels is dense in $A^p(\omega ^{p/2})$, $1\le p<\infty$.
\end{coro}
\begin{proof}
Since $E$ is a linear subspace of $A^p(\omega ^{p/2})$, by standard functional analysis and the duality results in Theorems \ref{Dual-Ber} and \ref{dual1}, it is enough to prove that $g\equiv 0$ if $g\in A^{p'}(\omega^{p'/2})$ satisfies $\langle f,g \rangle _{\omega}=0$ for each $f$ in $E$ (with $p'$ being the conjugate exponent of $p$, and $g\in A^{\infty}(\omega ^{1/2})$ if $p=1$). But, taking $f=K_ z$ for each $z\in \D$ and using the reproducing formula in Theorem \ref{RF-v1}, we get
$g(z)=P_{\omega} g(z)=\langle g,K_ z \rangle _{\omega}=0,$ for each $z\in \D$.
This finishes the proof.
\end{proof}
The next result identifies the predual of $A^{1}(\omega ^{1/2})$.
\begin{theorem}\label{dual0}
Let $\omega\in \mathcal{E}$. Under the integral pairing $\langle f,g \rangle_{\omega}$, the dual space of  $A_{0}(\omega ^{1/2})$ can be identified (with equivalent norms) with $A^{1}(\omega ^{1/2})$.
\end{theorem}
\begin{proof}
If  $g\in A^{1}(\omega^{1/2}),$ clearly
$\Lambda_{g}(f)= \langle f,g\rangle _{\omega}$ defines a bounded linear functional in $A_{0}(\omega ^{1/2})$ with $\|\Lambda _ g \|\le \|g\|_{A^{1}(\omega ^{1/2})}$. Conversely, assume that $\Lambda \in \big (A_ 0(\omega ^{1/2})\big )^*$. Consider the space $X$ that consists of functions of the form $h=f\omega^{1/2}$ with $f\in A_ 0(\omega ^{1/2})$. Clearly $X$ is a subspace of $C_ 0(\D)$ (the space of all continuous functions vanishing at the boundary) and $T(h)=\Lambda (\omega^{-1/2} h)=\Lambda (f)$ defines a bounded linear functional on $X$ with $\|T\|=\|\Lambda\|$. By Hahn-Banach, $T$ has an extension $\widetilde{T}\in \big (C_ 0(\D)\big )^*$  with $\|\widetilde{T}\|=\|T\|$. Hence, by Riesz representation theorem, there is a measure $\mu\in \mathcal{M}(\D)$ (the Banach space of all complex Borel measures $\mu$ equipped with the variation norm $\|\mu\|_{_\mathcal{M}}$) with $\|\mu\|_{_\mathcal{M}}=\|T\|$ such that
$$T(h)=\widetilde{T}(h)=\int_{\D} h(\zeta)\,d\mu(\zeta),\qquad h\in X,$$
or
$$\Lambda (f)=\int_{\D} f(\zeta)\,\omega(\zeta)^{1/2}\,d\mu(\zeta),\qquad f\in A_ 0(\omega^{1/2}).$$
Consider the function $g$ defined on the unit disk by
$$\overline{g(z)}=\int_{\D} K_ z(\zeta)\,\omega(\zeta)^{1/2}\,d\mu(\zeta),\qquad z\in \D.$$
Clearly $g$ is analytic on $\D$ and, by Fubini's theorem and Lemma \ref{kernelestimat}, we have
\begin{displaymath}
\begin{split}
\|g \|_{A^1(\omega^{1/2})} &\le \int_{\D} \left (\int_{\D} |K_ z(\zeta)|\,\omega (\zeta)^{1/2}\,d |\mu|(\zeta)\right ) \omega (z)^{1/2} dA(z)
\\
& =\int_{\D} \left (\int_{\D} |K_{\zeta}(z)| \,\omega (z)^{1/2}\,dA(z) \right ) \omega (\zeta)^{1/2}\,d |\mu|(\zeta)
\\
&\lesssim |\mu| (\D)=\|\mu\|_{_\mathcal{M}}=\|\Lambda\|,
\end{split}
\end{displaymath}
proving that $g$ belongs to $A^1(\omega ^{1/2})$. Now, since $A_ 0(\omega ^{1/2})\subset A^2(\omega)$, the reproducing formula $f(\zeta)=\langle f, K_ {\zeta} \rangle _{\omega}$ holds for all $f\in A_{0}(\omega ^{1/2})$. This and Fubini's theorem yields
\begin{displaymath}
\begin{split}
\Lambda_ g(f)=\langle f,g \rangle _{\omega}&=\int_{\D} f(z)\,\left (\int_{\D} K_ z(\zeta)\,\omega (\zeta)^{1/2}\,d\mu(\zeta) \right ) \,\omega (z) \,dA(z)
\\
&=\int_{\D} \left ( \int_{\D} f(z)\,\overline{K_{\zeta}(z)} \,\omega (z)\,dA(z)\right )\omega(\zeta)^{1/2}\, d\mu(\zeta)
\\
&=\int_{\D} f(\zeta)\,\omega(\zeta)^{1/2}\, d\mu(\zeta)=\Lambda (f).
\end{split}
\end{displaymath}
By Lemma \ref{LA0} and the reproducing formula in Theorem \ref{RF-v1}, the function $g$ is uniquely determined by the identity $g(z)=\overline{\Lambda (K_ z)}$. This completes the proof.
\end{proof}
For the case of normal weights, the analogues of Theorems \ref{dual1} and \ref{dual0} were obtained by Shields and Williams in \cite{ShW}. They also asked what happens with the exponential weights, problem that is solved in the present paper.

\section{Concluding remarks}

There is still plenty of work to do for a better understanding of the theory of large weighted Bergman spaces, and several natural problems are waiting for a further study or a complete solution: atomic decomposition, coefficient multipliers, zero sets, etc. We hope that the methods developed here will be of some help in order to attach the previous mentioned problems.

\end{document}